%% file: HOMFLYPT-unlinks-corrected.tex
\documentclass{amsart}
\usepackage{amssymb,latexsym,amsmath,amscd,graphicx,graphics,epic,eepic,bm,color,array,mathrsfs,fullpage}
\usepackage{enumerate}
\usepackage[all,knot,poly]{xy}

\newcommand{\zed}{\mathbb{Z}}

\newcommand{\Q}{\mathbb{Q}}

\newcommand{\bn}[2]{\genfrac{(}{)}{0pt}{}{#1}{#2}}

\newcommand{\id}{\mathrm{id}}

\newcommand{\gdim}{\mathrm{gdim}}

\newcommand{\slmf}{\mathfrak{sl}}

\theoremstyle{plain}
\newtheorem{theorem}{Theorem}[section]
\newtheorem{lemma}[theorem]{Lemma}

\newtheorem{corollary}[theorem]{Corollary}

\theoremstyle{definition}
\newtheorem{definition}[theorem]{Definition}
\newtheorem{example}[theorem]{Example}

\theoremstyle{remark}
\newtheorem{remark}[theorem]{Remark}

\numberwithin{equation}{section}

\begin{document}

\title{HOMFLYPT Homology over $\mathbb{Z}_2$ Detects Unlinks}

\author{Hao Wu}

\thanks{The author was partially supported by NSF grant DMS-1205879.}

\address{Department of Mathematics, The George Washington University, Phillips Hall, Room 739, 801 22nd Street NW, Washington DC 20052, USA. Telephone: 1-202-994-0653, Fax: 1-202-994-6760}

\email{haowu@gwu.edu}

\subjclass[2010]{Primary 57M27}

\keywords{HOMFLYPT homology, Rasmussen spectral sequence, Khovanov homology, unknot, unlink} 

\begin{abstract}
We apply the Rasmussen spectral sequence to prove that the $\mathbb{Z}^3$-graded vector space structure of the HOMFLYPT homology over $\mathbb{Z}_2$ detects unlinks. Our proof relies on a theorem of Batson and Seed stating that the $\mathbb{Z}^2$-graded vector space structure of the Khovanov homology over $\mathbb{Z}_2$ detects unlinks. 
\end{abstract}

\maketitle

\section{Introduction}\label{sec-intro}

If the Jones polynomial determines certain properties of links, then so does the HOMFLYPT polynomial. This is because the latter specializes to the former. A reasonable question is whether the same is true for categorifications of these link polynomials. The answer is far less clear since there is no simple ``specialization" from the HOMFLYPT homology to the Khovanov homology. Instead, we have the Rasmussen spectral sequence. In this manuscript, we study a particular problem of this type: unlink detection. 

Kronheimer and Mrowka proved in \cite{Kronheimer-Mrowka-khovanov-unknot} that the vector space structure of the Khovanov homology over $\zed_2=\zed/2\zed$ detects the unknot. Later, Hedden and Ni proved in \cite{Hedden-Ni-khovanov-module} that the module structure of the Khovanov homology over $\zed_2$ detects unlinks. Improving these results,  Batson and Seed proved in \cite{Batson-Seed} that the $\zed^2$-graded vector space structure of the Khovanov homology over $\zed_2$ detects unlinks.

Upon closer inspection of the Rasmussen spectral sequence, we found that, if the HOMFLYPT homology of a link is isomorphic to that of an unlink, then its Khovanov homology is also isomorphic to that of an unlink. In particular, the HOMFLYPT homology over $\zed_2$ detects unlinks. The following is our main result.

\begin{theorem}\label{thm-unlinks}
Suppose that $B$ is a closed braid of $m$ components. Let $R$ be the polynomial ring over $\zed_2$ generated by $m$ variables, graded so that each variable is homogeneous of degree $2$. Then the following two statements are equivalent.
\begin{enumerate}
	\item $B$ is Markov equivalent to the $m$-component unlink.
	\item As $\zed$-graded vector spaces over $\zed_2$, the unreduced HOMFLYPT homology $H(B)$ of $B$ over $\zed_2$ is given by 
	\[
	H^{i,j}(B) \cong \begin{cases}
	R^{\oplus \bn{m}{l}}\{2l\}_x & \text{if }i=m-2l, ~0\leq l \leq m, \text{ and } j=-m, \\
	0 & \text{otherwise,}
	\end{cases}
	\]
	where $H^{i,j}(B)$ is the subspace of $H(B)$ with horizontal grading $i$ and vertical grading $j$, and $\{s\}_x$ means shifting the $x$-grading by $s$.
\end{enumerate}
\end{theorem}

The proof of Theorem \ref{thm-unlinks} in Section \ref{sec-detection-unlinks} is itself quite short. But to understand this proof, one needs to know some basic properties of the HOMFLYPT homology and the Rasmussen spectral sequence connecting it to the Khovanov homology. So, for the convenience of the reader, we include a brief review of these subjects in Appendix \ref{sec-HOMFLYPT}. We need to point out that the normalization of $H(B)$ in the current paper is given in Appendix \ref{sec-HOMFLYPT}, which is different from the ones in \cite{KR2,Ras2}.

Also, there is a slight refinement of Theorem \ref{thm-unlinks}, in which one does not assume $B$ has exactly $m$ components to start with. This is because, according to Theorem \ref{thm-hilbert} below, the $\zed^3$-graded vector space structure of $H(B)$ actually determines the number of components of $B$. We will state this refinement in Corollary \ref{cor-unlinks} below.

\section{Detection of Unlinks}\label{sec-detection-unlinks}

We prove Theorem \ref{thm-unlinks} in this section. Let us establish some technical lemmas first.

\begin{lemma}\label{lemma-x-square-trivial}
Let $B$ be a closed braid of $m$ components with a marking, in which $x_i$ is the variable assigned to a marked point on the $i$-th component for $i=1,\dots,m$. Then there is an obvious identification ${R_B \cong\zed_2[x_1,\dots,x_m]}$.\footnote{See Lemma \ref{lemma-quotient-base-ring} below.} For $i=1,\dots,m$, there is a homogeneous $R_B$-module homomorphism $H(B) \xrightarrow{h_i} H(B)$ that preserves the vertical grading, raises the horizontal grading by $2$, lowers the $x$-grading by $2$, and satisfies that $d_-\circ h_i + h_i \circ d_- = x_i^2 \cdot \id_{H(B)}$. 
\end{lemma}

\begin{proof}
With out loss of generality, we assume there is a marked point on the $i$-th component $B$ adjacent to the marked point to which $x_i$ is assigned. That is, there is a segment $A$ in this component connecting these marked points with no crossings or other marked points in its interior. Say $y_i$ is the variable assigned to this adjacent marked points. Then the unnormalized HOMFLYPT chain complex of matrix factorizations associated to this segment $A$ is
\[
 \xymatrix{
 && 0 && 0 && \\
\hat{C}(A) = 0 \ar[rr]<0.5ex> && \ar[ll]<0.5ex> \zed_2[x_i,y_i]\{2\}_x \ar[rr]<0.5ex>^{x_i+y_i}\ar[u] && \zed_2[x_i,y_i]  \ar[ll]<0.5ex>^{x_i^2+x_iy_i+y_i^2} \ar[rr]<0.5ex> \ar[u] && \ar[ll]<0.5ex>0 .\\
&& 0 \ar[u] && 0 \ar[u] &&
}
\]
Define a homogeneous $\zed_2[x_i,y_i]$-homomorphism $h_i: \hat{C}(A) \rightarrow \hat{C}(A)$ by the diagram

\[
 \xymatrix{
 && 0 && 0 && \\
0 \ar[rr]<0.5ex>  && \ar[ll]<0.5ex> \zed_2[x_i,y_i]\{2\}_x \ar[rr]<0.5ex>^{x_i+y_i}\ar[u]  && \zed_2[x_i,y_i]  \ar[ll]<0.5ex>^{x_i^2+x_iy_i+y_i^2} \ar[rr]<0.5ex> \ar[u]  && \ar[ll]<0.5ex>0 \\
&& 0 \ar[u] && 0 \ar[u] && \\ 
 && 0 && 0 && \\
0 \ar[rr]<0.5ex> \ar[uuurr] && \ar[ll]<0.5ex> \zed_2[x_i,y_i]\{2\}_x  \ar[uuurr]^1 \ar[rr]<0.5ex>^{x_i+y_i}\ar[u] && \zed_2[x_i,y_i]  \ar[ll]<0.5ex>^{x_i^2+x_iy_i+y_i^2} \ar[rr]<0.5ex> \ar[u]  \ar[uuurr] && \ar[ll]<0.5ex>0 .\\
&& 0 \ar[u] && 0 \ar[u] &&
}
\]
Clearly, $h_i$ satisfies the grading conditions in the lemma and commutes with both differentials $d_+$ and $d_v$. Moreover,  
\begin{equation}\label{eq-homotopy-local}
d_-\circ h_i + h_i \circ d_- = (x_i^2+x_iy_i+y_i^2) \cdot \id_{\hat{C}(A)}.
\end{equation}

Denote by $R$ the polynomial ring over $\zed_2$ generated by all the variables used in the marking of $B$. Recall that, in Definition \ref{def-complex-tangle}, $\hat{C}(B)$ is defined by a big tensor product. Now we tensor $h_i$ with the identity maps of all the other factors in that big tensor product. This extends $h_i$ to an $R$-linear endomorphism of double complex $(\hat{C}(B), d_+,d_v) \xrightarrow{h_i} (\hat{C}(B), d_+,d_v)$ satisfying the grading requirements in the lemma. After the necessary grading shift, we get an $R$-linear endomorphism of double complex $(C(B), d_+,d_v) \xrightarrow{h_i} (C(B), d_+,d_v)$ satisfying the grading requirements in the lemma. Recall that $H(B)=H(H(C(B),d_+),d_v)$. Since $h_i$ commutes with both differentials of $(C(B), d_+,d_v)$, it induces a homogeneous $R$-module endomorphism $H(B) \xrightarrow{h_i} H(B)$ satisfying the grading requirements in the lemma. Since $R_B$ is a quotient ring of $R$ and the $R$-action on $H(B)$ factors through $R_B$, $H(B) \xrightarrow{h_i} H(B)$ is also an $R_B$-module homomorphism. Finally, by Lemma \ref{lemma-quotient-base-ring}, the multiplications of $x_i$ and $y_i$ are the same on $H(B)$. Thus, by equation \eqref{eq-homotopy-local}, $d_-\circ h_i + h_i \circ d_- = x_i^2 \cdot \id_{H(B)}$.
\end{proof}

A key ingredient of the proof of Theorem \ref{thm-unlinks} is the following algebraic lemma.

\begin{lemma}\label{lemma-homology-determination}
Let $\mathbb{F}$ be a field and $R=\mathbb{F}[t_1,t_2,\dots,t_m]$ the polynomial ring with $m$ variables over $\mathbb{F}$ graded so that each $t_i$ is homogeneous of degree $1$. Fix a positive integer $N$. For $0 \leq k \leq m$, denote by $M_k$ the graded $R$-module $M_k:= \mathbb{F}[t_1,t_2,\dots,t_m]/(t_{k+1}^N,\dots,t_m^N)$. Assume that the chain complex 
\[
(C_\ast,d)= 0 \rightarrow C_m \xrightarrow{d_m} C_{m-1} \xrightarrow{d_{m-1}} \cdots \xrightarrow{d_2} C_1 \xrightarrow{d_1} C_0 \rightarrow 0
\]
satisfies:
\begin{enumerate}
  \item $C_j$ is a graded $R$-module for $j=0,1,\dots,m$, 
  \item $C_j \cong M_k^{\oplus \bn{m}{j}}\{jN\}$ as graded $\mathbb{F}$-spaces for $j=0,1,\dots,m$, where $\{s\}$ means shifting the module grading up by $s$,
	\item $d_j$ is a homogeneous $R$-module homomorphism of degree $0$ for $j=1,\dots,m$,
	\item for each $i=1,\dots,k$, there is an $R$-module homomorphism $C_\ast \xrightarrow{h_i} C_{\ast+1}$ such that 
	\[
	d\circ h_i + h_i \circ d = t_i^N \cdot \id_{C_\ast}.
	\]
\end{enumerate}
Then 
\begin{equation}\label{eq-induction-k}
H_j(C_\ast, d) \cong 0 \text{ if } j>m-k.
\end{equation}
In particulart, if $k=m$, then
\begin{equation}\label{eq-induction-k=m}
H_j(C_\ast, d) \cong \begin{cases}
M_0 & \text{as graded } \mathbb{F}\text{-spaces if } j=0,\\
0 & \text{if } j \geq 1,
\end{cases}
\end{equation}
where $M_0$ is the graded $R$-module $M_0:=\mathbb{F}[t_1,t_2,\dots,t_m]/(t_{1}^N,t_{2}^N,\dots,t_m^N)$.
\end{lemma}

\begin{proof}
We prove Isomorphism \eqref{eq-induction-k} for $0 \leq k \leq m$ by inducting on $k$. If $k=0$, it is obvious that $H_j(C_\ast, d) \cong 0$ if $j>m$ since $C_j\cong 0$ if $j>m$. So Isomorphism \eqref{eq-induction-k} is true if $k=0$.

Now assume Isomorphism \eqref{eq-induction-k} is true for $k-1$. Assume $(C_\ast, d)$ satisfies all the conditions in the lemma. Since $d$ is an $R$-module homomorphism, $t_k^N \cdot C_\ast$ is a subcomplex of $C_\ast$. Denote by $D_\ast$ the quotient complex $D_\ast:= C_\ast/ t_k^N \cdot C_\ast$ and by $C_\ast \xrightarrow{\pi} D_\ast$ the standard quotient map. There is a short exact sequence
\begin{equation}\label{eq-short-exact}
0\rightarrow C_\ast \xrightarrow{t_k^N} C_\ast \xrightarrow{\pi} D_\ast \rightarrow 0,
\end{equation}
which induces a long exact sequence
\begin{equation}\label{eq-long-exact}
\cdots \xrightarrow{\Delta_{j+1}} H_j(C_\ast) \xrightarrow{t_k^N} H_j(C_\ast) \xrightarrow{\pi} H_j(D_\ast) \xrightarrow{\Delta_{j}} H_{j-1}(C_\ast) \xrightarrow{t_k^N} \cdots
\end{equation}
A simple diagram chase shows that the connecting homomorphism $\Delta_j$ is a homogeneous $R$-module homomorphism. By condition (3), the action by $t_k^N$ on $H(C_\ast)$ is $0$. So the long exact sequence \eqref{eq-long-exact} breaks into short exact sequences
\begin{equation}\label{eq-long-exact-broken}
0 \rightarrow H_j(C_\ast) \xrightarrow{\pi} H_j(D_\ast) \xrightarrow{\Delta_{j}} H_{j-1}(C_\ast) \rightarrow 0,
\end{equation}
where both $\pi$ and $\Delta_j$ are homogeneous $R$-homomorphisms. It is clear that $D_\ast$ is a complex satisfying, for $k-1$, all conditions in the lemma. So, by the induction hypothesis, $H_j(D_\ast) \cong 0$ if $j> m-k+1$. Then short exact sequences \eqref{eq-long-exact-broken} implies that $H_j(C_\ast) \cong 0$ if $j> m-k$. This completes the proof of Isomorphism \eqref{eq-induction-k}.

Now we focus on the case $k=m$ and prove Isomorphism \eqref{eq-induction-k=m}. By Isomorphism \eqref{eq-induction-k}, we have that, if $k=m$, then  $H_j(C_\ast) \cong 0$ when $j\geq 1$. It remains to show that, when $k=m$, $H_0(C_\ast) \cong M_0$ as graded $\mathbb{F}$-spaces. This comes down to a straightforward computation of the graded Euler characteristic. For any finitely generated graded $R$-module $M$, denote by $\gdim_\mathbb{F} M$ the graded dimension of $M$. That is, $\gdim_\mathbb{F} M = \sum_{\alpha\in \zed} X^\alpha \dim_\mathbb{F} M^\alpha$, where $M^\alpha$ is the homogeneous component of $M$ of degree $\alpha$. Since $H_j(C_\ast) \cong 0$ if $j\geq 1$, the graded Euler characteristic of the chain complex $C_\ast$ is 
\[
\gdim_\mathbb{F} H_0(C_\ast) = \sum_{j=0}^m (-1)^j \gdim_\mathbb{F} H_j(C_\ast) = \sum_{j=0}^m (-1)^j \gdim_\mathbb{F} C_j = \sum_{j=0}^m (-1)^j X^{jN} \bn{m}{j} \gdim_\mathbb{F} M_m.
\]
But $M_m \cong R$ and $\gdim_\mathbb{F} R = \frac{1}{(1-X)^m}$. So 
\[
\gdim_\mathbb{F} H_0(C_\ast) = \frac{1}{(1-X)^m}\sum_{j=0}^m (-1)^j X^{jN} \bn{m}{j} = \frac{(1-X^N)^m}{(1-X)^m} = \gdim_\mathbb{F} M_0.
\]
This shows that $H_0(C_\ast) \cong M_0$ as graded $\mathbb{F}$-spaces.
\end{proof}

In the proof of Theorem \ref{thm-unlinks}, we only need the following special case of Lemma \ref{lemma-homology-determination}.

\begin{corollary}\label{cor-homology-determination}
Let $R=\zed_2[x_1,x_2,\dots,x_m]$ the polynomial ring with $m$ variables over $\zed_2$ graded so that each $x_i$ is homogeneous of degree $2$. Assume that the chain complex $(C_\ast,d)$ is of the form
\[
 0 \leftarrow C_{-m} \xleftarrow{d_{-m+2}} C_{-m+2} \xleftarrow{d_{-m+4}} \cdots \xleftarrow{d_{m-2l-2}} C_{m-2l-2} \xleftarrow{d_{m-2l}} C_{m-2l} \xleftarrow{d_{m-2l+2}} \cdots \xleftarrow{d_{m-2}} C_{m-2} \xleftarrow{d_{m}} C_{m} \leftarrow 0,
\]
which satisfies:
\begin{enumerate}
  \item $C_{m-2l} \cong R^{\oplus \bn{m}{l}}\{2l\}_x$ as graded $R$-modules for $l=0,1,\dots,m$,  
	\item $d_{m-2l}$ is a homogeneous $R$-module homomorphism of degree $6$ for $l=0,\dots,m-1$,
	\item for each $i=1,\dots,m$, there is an $R$-module homomorphism $C_\ast \xrightarrow{h_i} C_{\ast+2}$ such that 
	\[
	d\circ h_i + h_i \circ d = x_i^2 \cdot \id_{C_\ast}.
	\]
\end{enumerate}
Then 
\[
H_{m-2l}(C_\ast, d) \cong \begin{cases}
\zed_2[x_1,x_2,\dots,x_m]/(x_i^2,\dots,x_m^2)\{2m\}_x & \text{as graded } \zed_2\text{-spaces if } l=m,\\
0 & \text{otherwise.}
\end{cases}
\]
\end{corollary}

\begin{proof}
This is the $\mathbb{F}=\zed_2$, $N=2$, $k=m$ case of Lemma \ref{lemma-homology-determination} with an alternative grading convention.
\end{proof}

Now we are ready to prove Theorem \ref{thm-unlinks}.

\begin{proof}[Proof of Theorem \ref{thm-unlinks}]
Fix a marking of $B$. Pick a marked point on each component of $B$. Assume $x_i$ is the variable assigned to the chosen marked points on the $i$-th component. Then there is an obvious identification $R_B\cong\zed_2[x_1,\dots,x_m]$. 

By Theorem \ref{thm-HOMFLYPT} and Example \ref{eg-unlink-computation}, we know that statement (1) implies statement (2). 

Now assume that statement (2) is true. Then, by Theorem \ref{thm-Ras-ss} and Lemma \ref{lemma-x-square-trivial}, $(E_1(B),d_1) \cong (H(B),d_-)$ is a chain complex satisfying all conditions in Corollary \ref{cor-homology-determination}. This implies that 
\[
E_2^{i,j}(B) \cong \begin{cases}
\zed_2[x_1,x_2,\dots,x_m]/(x_i^2,\dots,x_m^2)\{-m\}_q & \text{as graded } \zed_2\text{-spaces if } i=j=-m,\\
0 & \text{otherwise.}
\end{cases}
\]
Here, note the difference between the $x$-grading and the quantum grading. By Theorem \ref{thm-Ras-ss}, $d_k$ strictly lowers the vertical grading for all $k\geq 2$. But $E_2(B)$ is supported on a single vertical grading. This means $d_k=0$ for all $k\geq 2$. Thus, 
\[
E_\infty^{i,j}(B) \cong E_2^{i,j}(B) \cong \begin{cases}
\zed_2[x_1,x_2,\dots,x_m]/(x_i^2,\dots,x_m^2)\{-m\}_q & \text{as graded } \zed_2\text{-spaces if }  i=j=-m,\\
0 & \text{otherwise.}
\end{cases}
\]
Using Theorem \ref{thm-Ras-ss} again, we get that
\[
H_2^i(B) \cong \begin{cases}
	\zed_2[x_1,\dots,x_m]/(x_1^2,\dots,x_m^2)\{-m\}_q & \text{as graded } \zed_2\text{-spaces if } i=0, \\
	0 & \text{otherwise.}
	\end{cases}
\]
By \cite[Theorem 1.3]{Batson-Seed}, this implies that $B$ is Markov equivalent to the $m$-component unlink.
\end{proof}

\begin{remark}
Using Lemma \ref{lemma-homology-determination}, one can also prove that, if the HOMFLYPT homology over $\Q$ of a closed braid $B$ is isomorphic to that of the $m$-component unlink, then the $\slmf(N)$ homology over $\Q$ of $B$ is also isomorphic to that of the $m$-component unlink.
\end{remark}

\section{A Slight Refinement of Theorem \ref{thm-unlinks}}\label{sec-detection-unknot}

Our refinement of Theorem \ref{thm-unlinks} is based on an earlier result of the author in \cite{Wu-hilbert} that the HOMFLYPT homology determines the number of components of a link. More precisely, the degree of the Hilbert polynomial of the HOMFLYPT homology is the number of components of the link minus one. Since this result is stated over $\Q$ in \cite{Wu-hilbert}, for the convenience of the reader, we re-state it over $\zed_2$ and provide a sketch of its proof.

First, let us recall the Hilbert polynomial, which is a direct consequence of Hilbert's Syzygy Theorem. We state this theorem below. 

\begin{theorem}[Hilbert's Syzygy Theorem]\label{thm-syzygy}
Let $\mathbb{F}$ be a field and $R$ the polynomial ring $R=\mathbb{F}[x_1,\dots,x_m]$ graded so that each $x_i$ is homogeneous of degree $2$. We say that the grading of a graded $R$-module is even if this module contains no non-zero homogeneous elements of odd degrees.

Assume that $M$ is a finitely generated graded $R$-module whose grading is even. Denote by $M_{2n}$ the homogeneous component of $M$ of degree $2n$. Then there is an exact sequence of graded $R$ modules 
\[
0 \rightarrow F_l \rightarrow F_{l-1} \rightarrow \cdots \rightarrow F_1 \rightarrow F_0 \rightarrow M \rightarrow 0,
\] 
in which,
\begin{itemize}
  \item $l\leq m$,
	\item each $F_j$ is a finitely generated free graded module over $R$ whose grading is even,
	\item each arrow is a homogeneous homomorphism of $R$-modules preserving the grading.
\end{itemize}

As a standard consequence, there is a polynomial $P(T)\in \Q[T]$ of degree at most $m-1$ such that $\dim_\mathbb{F} M_{2T} = P(T)$ for $T\gg1$. This $P(T)$ is called the Hilbert polynomial of $M$.
\end{theorem}

For a proof of Theorem \ref{thm-syzygy}, see for example \cite{Peeva-book}. Note that, although the graded module structure of $M$ is used in the construction of its Hilbert polynomial, this polynomial is completely determined by the graded $\mathbb{F}$-space structure of $M$.

\begin{definition}\label{def-Hilbert}
Let $B$ be a closed braid and $R_B$ the polynomial ring over $\zed_2$ generated by components of $B$, graded so that each component of $B$ is homogeneous of degree $2$. From its definition in Appendix \ref{sec-HOMFLYPT}, the $x$-grading of $H(B)$ is even. For each horizontal grading $i$ and vertical grading $j$, we define $P_{B,i,j}(T)$ to be the Hilbert polynomial of the graded $R_B$-module $H^{i,j}(B)$ with respect to its $x$-grading. The Hilbert polynomial of $H(B)$ is $P_B(T)=\sum_{(i,j)\in \zed\times \zed}P_{B,i,j}(T)$. Here, note that the right hand side is a finite sum since $H(B)$ is finitely generated over $R_B$.
\end{definition}

The following theorem describes how the Hilbert polynomial of the HOMFLYPT homology determines the number of components of a link.

\begin{theorem}\cite[Theorem 1.2]{Wu-hilbert}\label{thm-hilbert}
Suppose that $B$ is a closed braid of $m$ components. Then the Hilbert polynomial $P_B(T)$ of $H(B)$ is of degree $m-1$.
\end{theorem}

Theorem \ref{thm-hilbert} is proved in \cite{Wu-hilbert} over $\Q$. Its proof over $\zed_2$ is essentially the same. Here we only include a sketch of its proof over $\zed_2$. For more details, see \cite{Wu-hilbert}.

\begin{proof}[Sketch of proof of Theorem \ref{thm-hilbert}]
First of all, $H(B)$ is a finitely generated graded $R_B$-module with an even $x$-grading. So, according to Theorem \ref{thm-syzygy}, $\deg_T P_B(T) \leq m-1$.

Now consider the polynomials 
\begin{eqnarray*}
\hat{F}_B(a,x) & = & \sum_{(i,j)\in \zed\times \zed} (-1)^{\frac{j}{2}} a^i x^{2T} \dim_{\zed_2} \hat{H}^{i,j,2T}(B), \\
\hat{Q}_B(a,T) & = & \sum_{(i,j)\in \zed\times \zed} (-1)^{\frac{j}{2}}a^i \hat{P}_{B,i,j}(T),
\end{eqnarray*} 
where 
\begin{itemize}
	\item $\hat{H}^{i,j,2T}(B)$ is the homogeneous component of the unnormalized HOMFLYPT homology $\hat{H}(B)$ of horizontal grading $i$, vertical grading $j$ and $x$-grading $2T$,
	\item $\hat{P}_{B,i,j}(T)$ is the Hilbert polynomial of $\hat{H}^{i,j}(B)$ with respect to the $x$-grading. 
\end{itemize}

\begin{figure}[ht]
\[
\xymatrix{
\input{B+} && \input{B-} && \input{B0}
}
\]
\caption{}\label{fig-skein}

\end{figure}

$\hat{F}_B(a,x)$ is the normalization of the HOMFLYPT polynomial satisfying 
\begin{equation}\label{eq-HOMFLYPT-normalization}
\begin{cases}
\hat{F}_B(a,x) \text{ is invariant under transverse Markov moves,} \\
xa^{-2} F_{B_+}(a,x) - x^{-1} a^2 F_{B_-}(a,x) = (x^{-1}-x)F_{B_0}(a,x), \\
F_{B'}(a,x)=-a^{-2}F_B(a,x), \\
F_U(a,x) = \frac{1+a^{-2}x^2}{1-x^2},
\end{cases}
\end{equation}
where 
\begin{itemize}
  \item transverse Markov moves are all Markov moves except the negative stabilization/destabilization,
	\item $B_+$, $B_-$ and $B_0$ are closed braids identical outside the part shown in Figure \ref{fig-skein},
	\item $B'$ is obtained from the closed braid $B$ by a negative stabilization,
	\item $U$ is the closed braid with a single strand.
\end{itemize}
From this, we know that $\hat{Q}_B(a,T)$ satisfies the skein relation: 
	\begin{equation}\label{eq-Hilbert-skein}
	\begin{cases}
	\hat{Q}_B(a,x) \text{ is invariant under transverse Markov moves,} \\
	a^{-2} \hat{Q}_{B_+}(a, T) - a^2 \hat{Q}_{B_-}(a, T+1) = \hat{Q}_{B_0}(a, T+1)- \hat{Q}_{B_0}(a, T),\\
	\hat{Q}_{B'}(\alpha, T) = -a^{-2} \hat{Q}_{B}(a, T), \\
	\hat{Q}_{U}(a, T) = 1+a^{-2}. \\
	\end{cases}
	\end{equation}

Using skein relation \eqref{eq-Hilbert-skein}, one can apply the ``computation tree argument" in \cite{FW} to $\hat{Q}(a,T)$ to establish that $\deg_T \hat{Q}(a,T) \geq m-1$. See \cite[Section 4]{Wu-hilbert} for more details. This implies that $\deg_T \hat{P}_{B,i,j}(T) \geq m-1$ for some $(i,j)\in \zed\times \zed$. But $H(B)$ is obtained from $\hat{H}(B)$ by shifting the horizontal and vertical gradings. So $\{\hat{P}_{B,i,j}(T)~|~ (i,j)\in \zed\times \zed\}$ and $\{P_{B,i,j}(T)~|~ (i,j)\in \zed\times \zed\}$ are the same collection of polynomials of $T$. So $\deg_T P_{B,i,j}(T) \geq m-1$ for some $(i,j)\in \zed\times \zed$. Since the leading coefficient of the Hilbert polynomial of a graded module must be positive, this implies that 
\[
\deg_T P_B(T) = \deg_T \sum_{(i,j)\in \zed\times \zed}P_{B,i,j}(T) = \max\{\deg_T P_{B,i,j}(T)~|~ (i,j)\in \zed\times \zed\} \geq m-1.
\]

Altogether we have $\deg_T P_B(T) = m-1$.
\end{proof}

Using Theorem \ref{thm-hilbert}, we get a slightly stronger version of the detection theorem of unlinks. Note that, unlike in Theorem \ref{thm-unlinks}, we do not assume $B$ has exactly $m$ components in Corollary \ref{cor-unlinks}.

\begin{corollary}\label{cor-unlinks}
Suppose that $B$ is a closed braid. Let $R$ be the polynomial ring over $\zed_2$ generated by $m$ variables, graded so that each variable is homogeneous of degree $2$. Then the following two statements are equivalent.
\begin{enumerate}
	\item $B$ is Markov equivalent to the $m$-component unlink.
	\item As graded $\zed_2$-spaces, the unreduced HOMFLYPT homology $H(B)$ of $B$ is given by 
	\[
	H^{i,j}(B) \cong \begin{cases}
	R^{\oplus \bn{m}{l}}\{2l\}_x & \text{if }i=m-2l, ~0\leq l \leq m, \text{ and } j=-m, \\
	0 & \text{otherwise,}
	\end{cases}
	\]
	where $H^{i,j}(B)$ is the $\zed_2$-subspace of $H(B)$ with horizontal grading $i$ and vertical grading $j$, and $\{s\}_x$ means shifting the $x$-grading by $s$.
\end{enumerate}
\end{corollary}

\begin{proof}
By Theorem \ref{thm-unlinks}, statement (1) implies statement (2). Now assume statement (2) is true. The Hilbert polynomial of $R$ is $\bn{T+m-1}{m-1}$, which is of degree $m-1$ in $T$. From this, one can see that $\deg_T P_B(T) =m-1$. Therefore, by Theorem \ref{thm-hilbert}, $B$ has exactly $m$ components. Thus, by Theorem \ref{thm-unlinks}, statement (1) is true.
\end{proof}

\appendix

\section{HOMFLYPT Homology, Khovanov Homology and Rasmussen Spectral Sequence over $\zed_2$}\label{sec-HOMFLYPT}

In this appendix, we briefly review the unreduced HOMFLYPT homology over $\zed_2$  and the Rasmussen spectral sequence over $\zed_2$ connecting the HOMFLYPT homology to the Khovanov homology. We use a grading convention that differs from the one in \cite{Ras2}. We would like to emphasize that, although the statement of Theorem \ref{thm-unlinks} does not mention the module structure of $H(B)$, its proof does use this structure. So we will include the definition of this module structure in Lemma \ref{lemma-quotient-base-ring} and Theorem \ref{thm-HOMFLYPT} below.

\subsection{Matrix factorizations} In \cite{Ras2}, both the HOMFLYPT homology and the Khovanov homology are defined using matrix factorizations with an integer homological grading.

\begin{definition}\label{def-mf}
Let $R$ be a graded commutative ring with $1$. In this manuscript, we call the grading of $R$ and the grading of any graded $R$-module the \textbf{$x$-grading}. Let $w$ be a homogeneous element of $R$ of $x$-grading $N$. A graded matrix factorization of $w$ over $R$ is a diagram 
\[
\xymatrix{
0 \ar[r]<.5ex>^{d_+^{l-2}} & M_l \ar[l]<.5ex>^{d_-^{l}}\ar[r]<.5ex>^{d_+^l} & M_{l+2} \ar[l]<.5ex>^{d_-^{l+2}} \ar[r]<.5ex>^{d_+^{l+2}} & \ar[l]<.5ex>^{d_-^{l+2k+4}}\cdots \ar[r]<.5ex>^{d_+^{l+2k-4}} & M_{l+2k-2} \ar[l]<.5ex>^{d_-^{l+2k-2}}\ar[r]<.5ex>^{d_+^{l+2k-2}} & M_{l+2k} \ar[l]<.5ex>^{d_-^{l+2k}} \ar[r]<.5ex>^{d_+^{l+2k}}  & \ar[l]<.5ex>^{d_-^{l+2k+2}} 0,
}
\]
where 
\begin{itemize}
	\item each $M_{l+2j}$ is a graded free $R$-module,
	\item each $d_+^{l+2j}$ is a homogeneous $R$-module homomorphism preserving the $x$-grading,
	\item each $d_-^{l+2j}$ is a homogeneous $R$-module homomorphism raising the $x$-grading by $N$,
	\item $d_+^{l+2j+2} \circ d_+^{l+2j} =0$, $d_-^{l+2j-2} \circ d_-^{l+2j} =0$ and $d_+^{l+2j-2}\circ d_-^{l+2j} + d_-^{l+2j+2}\circ d_+^{l+2j} = w \cdot \id_{M_{l+2j}}$ for $j=0,1,\dots, k$,
	\item the \textbf{horizontal grading} of this matrix factorization is defined so that $M_{l+2j}$ has horizontal grading $l+2j$ for $j=0,1,\dots, k$.
\end{itemize}
We call $d_+$ and $d_-$ the positive and negative differentials of the matrix factorization.
\end{definition}

In the definitions below, we will mostly use a special type of matrix factorizations -- Koszul matrix factorizations.

\begin{definition}\label{def-Koszul}
Let $a$ and $b$ be homogeneous elements of $R$. Then the graded Koszul matrix factorization $K_R(a,b)$ over $R$ of $ab$ is the following diagram
\[
\xymatrix{
K_R(r) := 0 \ar[r]<.5ex> & {\underbrace{R\{\deg b\}_x}_{-2}} \ar[l]<.5ex> \ar[rr]<.5ex>^>>>>>>>>>>{b} &&  {\underbrace{R}_{0}} \ar[ll]<.5ex>^<<<<<<<<<<{a} \ar[r]<.5ex> & \ar[l]<.5ex> 0,                                                   
}
\]
where $a$ and $b$ act by multiplication, the under-braces indicate the horizontal grading, and $\{s\}_x$ means shifting the $x$-grading up by $s$.

More generally, assume that $a_1,b_1,\dots,a_k,b_k$ are homogeneous elements of $R$ such that $a_1b_1,\dots,a_kb_k$ are of the same $x$-grading. Then the graded Koszul matrix factorization $K_R\left(%
\begin{array}{cc}
  a_1 & b_1 \\
	\cdots & \cdots \\
  a_k & b_k
\end{array}%
\right)$ over $R$ of $\sum_{j=1}^k a_j b_j$ is the tensor product over $R$ of the graded Koszul matrix factorizations of all rows $(a_j,b_j)$. That is,
\[
K_R\left(%
\begin{array}{cc}
  a_1 & b_1 \\
	\cdots & \cdots \\
  a_k & b_k
\end{array}%
\right) := K_R(a_1,b_1) \otimes_R K_R(a_2,b_2) \otimes_R \cdots \otimes_R K_R(a_k,b_k).
\]

Note that permuting the sequence $\{(a_1,b_1),\dots,(a_k,b_k)\}$ permutes the factors in the above tensor product and, therefore, does not change the isomorphism type of the Koszul matrix factorization $K_R\left(%
\begin{array}{cc}
  a_1 & b_1 \\
	\cdots & \cdots \\
  a_k & b_k
\end{array}%
\right)$.
\end{definition}

\subsection{Matrix factorizations of MOY graphs}

Next, we recall the matrix factorizations of MOY graphs used to defined both the HOMFLYPT homology and the Khovanov homology.

\begin{figure}[ht]
\[
\xymatrix{
\input{vertex}
}
\]
\caption{}\label{fig-MOY-vertex}

\end{figure}

\begin{definition}\label{def-MOY}
An MOY graph $\Gamma$ is an embedding of a directed graph in the plane so that 
\begin{itemize}
	\item each vertex of $\Gamma$ has valence $4$ or $1$,
	\item each $4$-valent vertex of $\Gamma$ looks like the one in Figure \ref{fig-MOY-vertex}.
\end{itemize}
We call a $1$-valent vertex of $\Gamma$ an endpoint, and a $4$-valent vertex an interior vertex.

A marking of an MOY graph $\Gamma$ consists of 
\begin{itemize}
	\item a finite collection of marked points on $\Gamma$ such that 
		\begin{itemize}
		\item none of the interior vertices are marked,
		\item all endpoints are marked,
		\item each edge of $\Gamma$ contains at least one marked points,
	  \end{itemize}
	\item an assignment to each marked point a distinct variable of degree $2$.
\end{itemize}
\end{definition}

\begin{figure}[ht]

\[
\xymatrix{
\input{arc_marked} && \input{wide-edge_marked}
}
\]

\caption{}\label{fig-MOY-pieces}

\end{figure}

Fix an MOY graph $\Gamma$ and a marking of $\Gamma$. Assume that $x_1,\dots,x_n$ are all the variables assigned to marked points in $\Gamma$. Let $R=\zed_2[x_1,\dots,x_n]$ graded so that $\deg x_1=\cdots=\deg x_n =2$. Cut $\Gamma$ at all of its marked points. We get a collection of pieces $\Gamma_1,\dots,\Gamma_m$, each of which is of one of the two types in Figure \ref{fig-MOY-pieces}. We define their matrix factorizations by the following.

\begin{itemize}
	\item If $\Gamma_q = \Gamma_{i;s}$ in Figure \ref{fig-MOY-pieces}, then $R_q = \zed_2[x_i,x_s]$ and $\hat{C}(\Gamma_q) = K_{R_q}(x_s^2+x_sx_i+x_i^2,x_s+x_i)$.
	\item If $\Gamma_q = \Gamma_{i,j;s,t}$ in Figure \ref{fig-MOY-pieces}, then $R_q = \zed_2[x_i,x_j,x_s,x_t]$ and
	\[
	\hat{C}(\Gamma_q) = K_{R_q}\left(%
\begin{array}{cc}
  x_s^2+x_t^2+x_i^2+x_j^2+x_t(x_s+x_i+x_j) & x_s+x_t+x_i+x_j \\
  x_i+x_j & (x_s+x_i)(x_s+x_j)
\end{array}%
\right)\{-1\}_x.
	\]
\end{itemize}

\begin{definition}\label{def-Koszul-MOY}
The matrix factorization of $\Gamma$ is defined to be
\[
\hat{C}(\Gamma) := \bigotimes_{q=1}^m (\hat{C}(\Gamma_q) \otimes_{R_q} R),
\]
where the big tensor product ``$\bigotimes_{q=1}^m$" is taken over the ring $R=\zed_2[x_1,\dots,x_n]$.

Note that $\hat{C}(\Gamma)$ is a graded matrix factorization over $R$. It has two gradings: the horizontal grading and the $x$-grading. Its positive differential $d_+$ raises the horizontal grading by $2$ and preserves the $x$-grading. Its negative differential $d_-$ lowers the horizontal gradings by $2$ and raises the $x$-grading by $6$.  
\end{definition}

\subsection{The unnormalized HOMFLYPT complex of a tangle} 

\begin{definition}\label{def-marking-braid}
For a tangle $T$, an arc of $T$ is a part of $T$ that starts and ends at crossings or endpoints and contains no crossings or endpoints in its interior. A marking of $T$ consists of
\begin{itemize}
	\item a finite collection of marked points on $T$ such that 
	\begin{itemize}
		\item all endpoints of $T$ are marked,
		\item none of the crossings of $T$ is marked, 
		\item each arc of $T$ contains at least one marked point,
	\end{itemize}
	\item an assignment that assigns to each marked point a distinct variable of degree $2$.
\end{itemize}
\end{definition}

Let $T$ be a tangle with a marking. Assume $x_1,\dots,x_n$ are all the variables assigned to marked points in $T$. The ring $R:=\zed_2[x_1,\dots,x_n]$ is graded so that $\deg x_j =2$ for $j=1,\dots,n$.

Cut $T$ at all of its marked points. This cuts $T$ into a collection $\{T_1,\dots,T_m\}$ of simple tangles, each of which is of one of the three types in Figure \ref{tangle-pieces-fig} and is marked only at its endpoints.

\begin{figure}[ht]
$
\xymatrix{
\input{arc} && \input{crossing+} && \input{crossing-} 
}
$
\caption{}\label{tangle-pieces-fig}

\end{figure}

If $T_q=A$, then $R_q =\zed_2[x_i,x_s]$ and the unnormalized HOMFLYPT complex $\hat{C}(T_q)$ of $T_q$ is 
\begin{equation}\label{eq-def-chain-arc}
\xymatrix{
{\underbrace{0}_{2}}  \\ {\underbrace{\hat{C}(A)}_{0}} \ar[u] \\ {\underbrace{0}_{-2}} \ar[u]},
\end{equation}
where $\hat{C}(A)$ is the matrix factorization defined in the previous subsection, and the under-braces indicate the \textbf{vertical gradings}.

To define the complexes of crossings, we need the $\chi$-maps.

\begin{figure}[ht]
$
\xymatrix{
\input{crossing-1-1-res-0} \ar@<8ex>[rr]^{\chi^0} && \input{crossing-1-1-res-1} \ar@<-6ex>[ll]^{\chi^1}
}
$
\caption{}\label{def-chi-fig}

\end{figure}

\begin{lemma}\cite{KR1,KR2,Ras2}\label{lemma-def-chi}
Let $\Gamma_0$ and $\Gamma_1$ be the MOY graphs marked as in Figure \ref{def-chi-fig}, and $R=\zed_2[x_s,x_t,x_i,x_j]$. Then there exist homogeneous morphisms of matrix factorizations $\hat{C}(\Gamma_0) \xrightarrow{\chi^0} \hat{C}(\Gamma_1)$ and $\hat{C}(\Gamma_1) \xrightarrow{\chi^1} \hat{C}(\Gamma_0)$ of $x$-degree $1$ satisfying $\chi^1 \circ \chi^0 \simeq (x_s-x_j)\cdot \id_{\hat{C}(\Gamma_0)}$ and $\chi^0 \circ \chi^1 \simeq (x_s-x_j)\cdot \id_{\hat{C}(\Gamma_1)}$. Here, a map between two matrix factorizations over $R$ is a ``morphism of matrix factorizations" if it is an $R$-module homomorphism preserving the horizontal grading and commuting with both the positive and the negative differentials. 
\end{lemma}

\begin{figure}[ht]
$
\xymatrix{
&& \input{crossing+}  \ar@<-12ex>[lld]\ar@<12ex>[rrd] && \\
 \input{crossing-1-1-res-0}&&&& \input{crossing-1-1-res-1} \\
&& \input{crossing-} \ar[llu] \ar[rru] &&
}
$
\caption{}\label{crossing-res-fig}

\end{figure}

Now consider the resolutions in Figure \ref{crossing-res-fig}. If $T_q=C_\pm$, then $R_q = \zed_2[x_i,x_j,x_s,x_t]$ and the unnormalized HOMFLYPT chain complex $\hat{C}(T_q)$ of $T_q$ is 
\[
\xymatrix{
\hat{C}(C_+) \ar@{=}[d] && \text{or} && \hat{C}(C_-) \ar@{=}[d] \\
{\underbrace{0}_{2}} && && {\underbrace{0}_{4}} \\
{\underbrace{\hat{C}(\Gamma_0)\{2\}_h\{-2\}_x}_{0}} \ar[u] &&  && {\underbrace{\hat{C}(\Gamma_1)\{-2\}_h\{1\}_x}_{2}}  \ar[u]\\
{\underbrace{\hat{C}(\Gamma_1)\{2\}_h\{-1\}_x}_{-2}} \ar[u]^<<<<<{\chi^1} &&  && {\underbrace{\hat{C}(\Gamma_0)\{-2\}_h\{2\}_x}_{0}}  \ar[u]^<<<<<{\chi^0} \\
{\underbrace{0}_{-4}}  \ar[u] && && {\underbrace{0}_{-2}}  \ar[u]
},
\]
where the under-braces indicate the \textbf{vertical gradings}, $\{s\}_h$ means shifting the horizontal grading by $s$, and $\{t\}_x$ means shifting the $x$-grading by $t$.

\begin{definition}\label{def-complex-tangle}
We define the unnormalized HOMFLYPT chain complex $\hat{C}(T)$ associated to $T$ to be 
\[
\hat{C}(T)) := \bigotimes_{q=1}^{m} (\hat{C}(T_q)\otimes_{R_q} R),
\]
where the big tensor product ``$\bigotimes_{q=1}^{m}$" is taken over $R=\zed_2[x_1,\dots,x_n]$. 

Note that $\hat{C}(T)$ has three gradings: the vertical grading, the horizontal grading and the $x$-grading. It has three differentials: 
\begin{itemize}
	\item the vertical differential $d_v$, which raises the vertical grading by $2$ and preserves the two other gradings,
	\item the positive differential $d_+$, which raises the horizontal grading by $2$ and preserves the two other gradings,
	\item the negative differential $d_-$, which preserves the vertical grading, lowers the horizontal grading by $2$ and raises the $x$-grading by $6$.
\end{itemize}
Moreover, these three differentials commute with each other.\footnote{Since we are working over $\zed_2$, there is no difference between commuting and anti-commuting.}  
\end{definition}

\subsection{HOMFLYPT homology} The HOMFLYPT homology of links are defined using closed braid diagrams of links. Let $B$ be a closed braid with a marking. Assume $B$ has $b$ strands and its writhe is $w$. Recall that the self linking number of $B$ is $sl(B)=w-b$.

\begin{definition}\label{def-homology-braid}
The unnormalized unreduced HOMFLYPT homology of $B$ is $\hat{H}(B):= H(H(\hat{C}(B),d_+),d_v)$.

The normalized HOMFLYPT chain complex of $B$ is $C(B) := \hat{C}(B)\{w-b\}_v\{b-w\}_h$, where $\{r\}_v\{s\}_h$ means shifting the vertical grading by $r$ and the horizontal grading by $s$. The unreduced HOMFLYPT homology of $B$ is $H(B):= H(H(C(B),d_+),d_v)=\hat{H}(B)\{w-b\}_v\{b-w\}_h$.
\end{definition}

Let $B$ be a closed braid with a marking, and $x_1,\dots,x_n$ all the variables assigned to marked points on $B$. The ring $R=\zed_2[x_1,\dots,x_n]$ is graded so that each $x_j$ is homogeneous of degree $2$. From its definition, $C(B)$ is a $\zed^3$-graded $R$-module, and all of its three differentials are homogeneous $R$-module homomorphisms. So $H(B):= H(H(C(B),d_+),d_v)$ is a $\zed^3$-graded $R$-module.

\begin{lemma}\cite[Lemma 3.4]{Ras-2-bridge}\label{lemma-quotient-base-ring}
If $x_i$ and $x_j$ are assigned to marked points on the same component of $B$, then the multiplications by $x_i$ and by $x_j$ are the same on $H(B)$.

In particular, this means that $H(B)$ is a finitely generated $\zed^3$-graded module over the graded ring
\[
R_B:= R/(\{x_i-x_j~|~x_i \text{ and } x_j \text{ are assigned to marked points on the same component of } B\}).
\]
\end{lemma}

The proof in \cite{Ras-2-bridge} of this lemma is for $\slmf(N)$ homology over $\Q$. But that proof can be adapted to work for $H(B)$ without essential changes. We leave details to the reader. 

Note that $R_B$ can be viewed as the polynomial ring over $\zed_2$ generated by components of $B$, with each component homogeneous of degree $2$.

\begin{theorem}\cite{KR2,Krasner-integral}\label{thm-HOMFLYPT}
Suppose that $B$ is a closed braid and $R_B$ is the polynomial ring over $\zed_2$ generated by components of $B$, graded so that each component of $B$ is homogeneous of degree $2$. Then the unreduced HOMFLYPT homology $H(B)$ of $B$ over $\zed_2$ is a finitely generated $\zed^3$-graded $R_B$-module. Up to isomorphisms of $\zed^3$-graded $R_B$-modules, $H(B)$ is independent of the marking and invariant under Markov moves. 
\end{theorem}

Theorem \ref{thm-HOMFLYPT} is proved over $\Q$ in \cite{KR2} and over $\zed$ in \cite{Krasner-integral}. These proofs can be adapted to the base field $\zed_2$ without essential changes. The $R_B$-module structure of the HOMFLYPT homology is not explicitly addressed in \cite{KR2,Krasner-integral}. But the proofs of invariance in these papers are local. For each local marking change or braid-like Reidemeister move, the isomorphism constructed in \cite{KR2,Krasner-integral} commutes with multiplications by variables assigned to marked points that are outside the interior of the part of $B$ affected by this local marking change or braid-like Reidemeister move. By the definition of markings of closed braids, there is always such a marked point on each component of $B$. This implies that the isomorphisms constructed in \cite{KR2,Krasner-integral} are isomorphisms of $R_B$-modules. Again, we leave details to the reader.

\subsection{Rasmussen spectral sequence} In \cite{Ras2}, Rasmussen proved that the spectral sequence associated to the column filtration of the double complex $(H(C(B),d_+),d_-,d_v)$ converges to the $\slmf(2)$ Khovanov homology. Of course, the proof in \cite{Ras2} is over $\Q$. But, again, this proof adapts to the base field $\zed_2$ without essential changes.

\begin{theorem}\cite{Ras2}\label{thm-Ras-ss}
Suppose that $B$ is a closed braid and $R_B$ is the polynomial ring over $\zed_2$ generated by components of $B$, graded so that each component of $B$ is homogeneous of degree $2$. Denote by $H_2(B)$ the Khovanov homology of $B$ over $\zed_2$. Then, the spectral sequence $\{E_r(B)\}$ associated to the column filtration of the double complex $(H(C(B),d_+),d_-,d_v)$ satisfies:
\begin{enumerate}
  \item $\{E_r(B)\}$ is a spectral sequence of $R_B$-modules.
	\item Each page of $\{E_r(B)\}$ inherits the horizontal and the vertical gradings of $C(B)$. It also inherits the $\slmf(2)$ quantum grading of $C(B)$ given by the degree function $\deg_x + 3 \deg_h$, where $\deg_x$ and $\deg_h$ are the degree functions of the $x$- and the horizontal gradings.
	\item $E_1(B) \cong H(B)$ as $R_B$-modules, where the isomorphism preserves the horizontal, vertical and $\slmf(2)$ quantum gradings.
	\item $E_\infty(B) \cong H_2(B)$ as $R_B$-modules, where the quantum grading of $H_2(B)$ coincide with the $\slmf(2)$ quantum grading of $E_\infty(B)$, and the homological grading of $H_2(B)$ coincides with the grading on $E_\infty(B)$ given by the degree function $\deg_v-\deg_h$, where $\deg_v$ and $\deg_h$ are the degree functions of the vertical and the horizontal gradings of $E_\infty(B)$.
	\item Under the identification $E_1(B) \cong H(B)$, the differential $d_1$ of the page $E_1(B)$ is the differential $d_-$ of $H(B)$, which preserves the vertical and the $\slmf(2)$ quantum gradings, and lowers the horizontal grading by $2$. For all $k\geq 2$, the differential $d_k$ of the pages $E_k(B)$ strictly lowers the vertical grading.
\end{enumerate}
\end{theorem}

Recently, Naisse and Vaz proved in \cite{Naisse-Vaz-Ras-ss} that the Rasmussen spectral sequence collapses at its $E_2$-page. For our purpose though, we do not need this general result. This is because that, if the HOMFLYPT homology of a link is isomorphic to that of an unlink, then its Rasmussen spectral sequence collapses at its $E_2$-page for simple grading reasons.

\begin{example}\label{eg-unlink-computation}
Let $U^{\sqcup m}$ be the $m$-strand closed braid with no crossings. This is of course a diagram of the $m$-component unlink. We mark $U^{\sqcup m}$ by putting a single marked points on each component of $U^{\sqcup m}$. Denote by $x_j$ the variable assigned to the marked point on the $j$-th component. Then $R_{U^{\sqcup m}}=\zed_2[x_1,\dots,x_m]$ and
\[
C(U^{\sqcup m}) = K_{R_{U^{\sqcup m}}}\left(%
\begin{array}{cc}
  x_1^2 & 0 \\
	\cdots & \cdots \\
  x_m^2 & 0
\end{array}%
\right)\{-m\}_v\{m\}_h.
\]
$d_v$ and $d_+$ both vanish for this complex. So $H(U^{\sqcup m}) \cong C(U^{\sqcup m})$ as $\zed^3$-graded $R_B$-modules. In other words, as $\zed$-graded $R_B$-modules, 
\[
	H^{i,j}(U^{\sqcup m}) \cong \begin{cases}
	R_{U^{\sqcup m}}^{\oplus \bn{m}{l}}\{2l\}_x & \text{if } i=m-2l, ~0\leq l \leq m, \text{ and }  j=-m, \\
	0 & \text{otherwise,}
	\end{cases}
\]
where $H^{i,j}(U^{\sqcup m})$ is the $R_{U^{\sqcup m}}$-submodule of $H(U^{\sqcup m})$ with horizontal grading $i$ and vertical grading $j$.

The complex $(H(U^{\sqcup m}), d_-)$ is the Koszul complex (with non-standard gradings) over $R_{U^{\sqcup m}}=\zed_2[x_1,\dots,x_m]$ for the sequence $\{x_1^2,\dots,x_m^2\}$, which is regular in $R_{U^{\sqcup m}}$. So 
\[
E_2^{i,j}(U^{\sqcup m})\cong H^{i,j}(H(U^{\sqcup m}), d_-) \cong \begin{cases}
	\zed_2[x_1,\dots,x_m]/(x_1^2,\dots,x_m^2)\{2m\}_x & \text{if} ~i=j=-m, \\
	0 & \text{otherwise,}
	\end{cases}
\]
where $E_2^{i,j}(U^{\sqcup m})$ is the $R_{U^{\sqcup m}}$-submodule of $E_2(U^{\sqcup m})$ with horizontal grading $i$ and vertical grading $j$. Since $d_k$ strictly lowers the vertical grading for $k \geq 2$, we have that $d_k=0$ for $k \geq 2$. So $E_\infty(U^{\sqcup m}) \cong E_2(U^{\sqcup m})$. Thus,
\[
H_2^i(U^{\sqcup m}) \cong \begin{cases}
	\zed_2[x_1,\dots,x_m]/(x_1^2,\dots,x_m^2)\{-m\}_q & \text{if} ~i=0, \\
	0 & \text{otherwise,}
	\end{cases}
\]
where $H_2^i(U^{\sqcup m})$ is the $R_{U^{\sqcup m}}$-submodule of $H_2(U^{\sqcup m})$ with homological grading $i$, and $\{-m\}_q$ means shifting the quantum grading down by $m$. Here, note the difference between the $x$-grading and the $\slmf(2)$ quantum grading.

In particular, when $m=1$, $U = U^{\sqcup 1}$ is the unknot. And
\begin{eqnarray*}
H^{i,j}(U) & \cong & {\begin{cases}
	\zed_2[x_1] & \text{if} ~i=1,~j=-1, \\
	\zed_2[x_1]\{2\}_x & \text{if} ~i=-1,~j=-1, \\
	0 & \text{otherwise,}
	\end{cases}} \\
H_2^i(U) & \cong & {\begin{cases}
	\zed_2[x_1]/(x_1^2)\{-1\}_q & \text{if} ~i=0, \\
	0 & \text{otherwise.}
	\end{cases}}
\end{eqnarray*}
\end{example}

\end{document}

%% file: B+.tex
\setlength{\unitlength}{1pt}
\begin{picture}(60,55)(-30,-15)

\put(-20,0){\vector(1,1){40}}

\put(20,0){\line(-1,1){18}}

\put(-2,22){\vector(-1,1){18}}

\put(-4,-15){$B_+$}

\end{picture}

%% file: B-.tex
\setlength{\unitlength}{1pt}
\begin{picture}(60,55)(-30,-15)

\put(20,0){\vector(-1,1){40}}

\put(-20,0){\line(1,1){18}}

\put(2,22){\vector(1,1){18}}

\put(-4,-15){$B_-$}

\end{picture}

%% file: B0.tex
\setlength{\unitlength}{1pt}
\begin{picture}(60,55)(-30,-15)

\put(-20,0){\vector(0,1){40}}

\put(20,0){\vector(0,1){40}}

\put(-4,-15){$B_0$}

\end{picture}

%% file: vertex.tex
\setlength{\unitlength}{1pt}
\begin{picture}(60,40)(-30,0)

\put(-20,0){\vector(1,1){20}}

\put(0,20){\vector(1,1){20}}

\put(20,0){\vector(-1,1){20}}

\put(0,20){\vector(-1,1){20}}

\end{picture}

%% file: arc_marked.tex
\setlength{\unitlength}{1pt}
\begin{picture}(60,45)(-30,-15)

\put(-20,15){\vector(1,0){40}}

\put(-30,15){\small{$x_i$}}

\put(23,15){\small{$x_s$}}

\put(-4,-15){$\Gamma_{i;s}$}

\end{picture}

%% file: wide-edge_marked.tex
\setlength{\unitlength}{1pt}
\begin{picture}(70,45)(-35,-15)

\put(-20,0){\vector(4,3){20}}

\put(-20,30){\vector(4,-3){20}}

\put(0,15){\vector(4,3){20}}

\put(0,15){\vector(4,-3){20}}

\put(-35,0){\small{$x_j$}}

\put(-35,25){\small{$x_i$}}

\put(28,0){\small{$x_t$}}

\put(28,25){\small{$x_s$}}

\put(-4,-15){$\Gamma_{i,j;s,t}$}

\end{picture}

%% file: arc.tex
\setlength{\unitlength}{1pt}
\begin{picture}(60,55)(-30,-15)

\put(0,0){\vector(0,1){40}}

\put(-10,35){\small{$x_s$}}

\put(-10,0){\small{$x_i$}}

\put(-4,-15){$A$}

\end{picture}

%% file: crossing+.tex
\setlength{\unitlength}{1pt}
\begin{picture}(60,55)(-30,-15)

\put(-20,0){\vector(1,1){40}}

\put(20,0){\line(-1,1){18}}

\put(-2,22){\vector(-1,1){18}}

\put(-30,35){\small{$x_s$}}

\put(-30,0){\small{$x_i$}}

\put(23,35){\small{$x_t$}}

\put(23,0){\small{$x_j$}}

\put(-4,-15){$C_+$}

\end{picture}

%% file: crossing-.tex
\setlength{\unitlength}{1pt}
\begin{picture}(60,55)(-30,-15)

\put(20,0){\vector(-1,1){40}}

\put(-20,0){\line(1,1){18}}

\put(2,22){\vector(1,1){18}}

\put(-30,35){\small{$x_s$}}

\put(-30,0){\small{$x_i$}}

\put(23,35){\small{$x_t$}}

\put(23,0){\small{$x_j$}}

\put(-4,-15){$C_-$}

\end{picture}

%% file: crossing-1-1-res-0.tex
\setlength{\unitlength}{1pt}
\begin{picture}(60,55)(-30,-15)

\put(-20,0){\vector(0,1){40}}

\put(20,0){\vector(0,1){40}}

\put(-30,35){\small{$x_s$}}

\put(-30,0){\small{$x_i$}}

\put(23,35){\small{$x_t$}}

\put(23,0){\small{$x_j$}}

\put(-4,-15){$\Gamma_0$}

\end{picture}

%% file: crossing-1-1-res-1.tex
\setlength{\unitlength}{1pt}
\begin{picture}(60,55)(-30,-15)

\put(-20,0){\vector(1,1){20}}

\put(20,0){\vector(-1,1){20}}

\put(0,20){\vector(-1,1){20}}

\put(0,20){\vector(1,1){20}}

\put(-30,35){\small{$x_s$}}

\put(-30,0){\small{$x_i$}}

\put(23,35){\small{$x_t$}}

\put(23,0){\small{$x_j$}}

\put(-4,-15){$\Gamma_1$}

\end{picture}

%% file: HOMFLYPT-unlinks-corrected.bbl
\begin{thebibliography}{99}
\bibliographystyle{plain} 
  \bibitem{Batson-Seed}
   J. Batson, C. Seed,
   \textit{A link-splitting spectral sequence in Khovanov homology,}
    Duke Math. J. \textbf{164} (2015), no. 5, 801--841. 
  \bibitem{FW}
   J. Franks, R. F. Williams,
   \textit{Braids and the Jones polynomial,}
   Trans. Amer. Math. Soc. \textbf{303} (1987), no. 1, 97--108.
	\bibitem{Hedden-Ni-khovanov-module}
   M. Hedden, Y. Ni,
   \emph{Khovanov module and the detection of unlinks,}
    Geom. Topol. \textbf{17} (2013), no. 5, 3027--3076.
	\bibitem{KR1}
   M. Khovanov, L. Rozansky,
   \emph{Matrix factorizations and link homology,}
   Fund. Math. \textbf{199} (2008), no. 1, 1--91.
  \bibitem{KR2}
   M. Khovanov, L. Rozansky,
   \emph{Matrix factorizations and link homology II,}
   Geom. Topol. \textbf{12} (2008), no. 3, 1387--1425.
	\bibitem{Krasner-integral}
   D. Krasner,
   \emph{Integral HOMFLY-PT and $\mathrm{sl}(n)$-link homology,}
   Int. J. Math. Math. Sci. 2010, Art. ID 896879, 25 pp. 
  \bibitem{Kronheimer-Mrowka-khovanov-unknot}
   P. Kronheimer, T. Mrowka,
   \emph{Khovanov homology is an unknot-detector,}
    Publ. Math. Inst. Hautes \'Etudes Sci. No. 113 (2011), 97--208. 
  \bibitem{Naisse-Vaz-Ras-ss}
   G. Naisse, P. Vaz,
   \emph{$2$-Verma modules and Khovanov-Rozansky link homologies,}
    arXiv:1704.08485. 
  \bibitem{Peeva-book}
   I. Peeva,
   \emph{Graded syzygies,}
   Algebra and Applications, 14. Springer-Verlag London, Ltd., London, 2011. xii+302 pp. ISBN: 978-0-85729-176-9 
  \bibitem{Ras-2-bridge}
   J. Rasmussen,
   \textit{Khovanov-Rozansky homology of two-bridge knots and links,}
   Duke Math. Journal \textbf{136} (2007), 551--583
  \bibitem{Ras2}
   J. Rasmussen,
   \textit{Some differentials on Khovanov-Rozansky homology,}
   Geom. Topol. \textbf{19} (2015) 3031--3104, DOI: 10.2140/gt.2015.19.3031
  \bibitem{Wu-hilbert}
   H. Wu,
   \textit{On the Hilbert polynomial of the HOMFLYPT homology,}
    J. Knot Theory Ramifications, \textbf{27} (2018), no. 1, 1850007 (20 pages)
\end{thebibliography}
